\documentclass[12pt,a4paper]{article}

\usepackage{latexsym}
\usepackage{amsfonts}
\usepackage{amssymb}
\usepackage{amsmath}
\usepackage{amsthm}
\usepackage{amscd}
\usepackage{mathrsfs}
\usepackage{enumerate}
\usepackage{slashed}
\usepackage{color}

\newcommand{\Prim}{\mathrm{Prim}}
\newcommand{\pt}{\mathrm{pt}}
\newcommand{\Sppd}{\mathrm{Sppd}\:}

\numberwithin{equation}{section}
\newtheorem{dfn}{Definition}
\newtheorem{lem}{Lemma}
\newtheorem{prp}{Proposition}
\newtheorem{thm}{Theorem}
\newtheorem*{te*}{Theorem}

\newtheorem*{cor}{Corollary}
\newtheorem*{dfn*}{Definition}

\begin{document}

\title{Locally $C^*$ Algebras, $C^*$ Bundles and Noncommutative Spaces}
 \author{Michael Forger%
         \thanks{Partly supported by CNPq
                 (Conselho Nacional de Desenvolvimento
                  Cient\'{\i}fico e Tecno\-l\'o\-gico), Brazil; \newline
                 E-mail: \textsf{forger@ime.usp.br}}
         ~~and~~
        Daniel V.\ Paulino%
        \thanks{Supported by FAPESP
                (Funda\c{c}\~ao de Amparo \`a Pesquisa do
                 Estado de S\~ao Paulo), Brazil; \newline
                 E-mail: \textsf{daniel.paulino@desy.de}}
        }
\date{\small{II Institute for Theoretical Physics, \\ University of Hamburg \\ - \\
             Departamento de Matem\'atica Aplicada \\
	     Instituto de Matem\'atica e Estat\'{\i}stica \\
	     Universidade de S\~ao Paulo}}
\maketitle

\thispagestyle{empty}

\begin{abstract}
 \noindent
 This work provides a generalization of the Gelfand duality to the context of 
 noncommutative locally $C^*$ algebras. 
 Using a reformulation of a theorem proven by Dauns and Hofmann in the 60's we 
 show that every locally $C^*$ algebra can be realized as the algebra of continuous 
 sections of a $C^*$ bundle over a compactly generated topological space.
 This result is used then to show that on certain special cases locally $C^*$
 algebras can be used to define certain sheaves of locally $C^*$ algebras that, 
 inspired by the analogy with commutative geometry, we call noncommutative spaces.
 The last section provides some examples, motivated by mathematical physics, for 
 this definition of noncommutative space. Namely we show that every local net of $C^*$ 
 algebras defines a noncommutative space and, based on a loose generalization of 
 the original construction by Doplicher, Fredenhagen and Roberts, construct what 
 we propose to call a "locally covariant quantum spacetime".
\end{abstract}

\newpage

\section{Introduction}

Perhaps the most well know result from the theory of $C^*$ algebras is the Gelfand duality, 
\cite{GN}, a categorical equivalence between commutative $C^*$ algebras and locally compact 
topological spaces.

The interest in extending this idea to noncommutative algebras, that is, to interpret noncommutative 
$C^*$ algebras as "function algebras" over some sort of ``noncommutative spaces", can perhaps be traced 
back to the early origins of quantum physics, where noncommutative $C^*$ algebras play a prominent role.
The attempts to give a precise meaning to the expression ``noncommutative space" led to the development 
of many different theories, some going as far as to define new branches of mathematics, like Connes' 
notion of spectral triple, \cite{CN}. 

The end goal of many such approaches is the definition of geometric structures, such as differential 
structures, pseudo-riemannian metrics, spin structures, etc, over those noncommutative spaces. However 
one point that is consistently overlooked in the current literature is the fact that, even in the classical 
setting, most of those constructions rely on the idea of "localization" of the structures involved.

By "localization" we mean the fact that there is a clear and well-defined notion of "sub-region" for the
spaces and that, in general, the information about constructions at the large can be obtained by the analysis 
of their restriction to smaller regions.

This is, in fact, a guiding principle in differential geometry, which becomes clear in the intuitive idea behind
the notion of a manifold; a space which can be reduced to "small regions" which are "similar" to regions of a 
euclidean space. 

This idea of "localization" is encoded in the categorical notion of sheaf, and it is a well-know fact that most 
of the usual differential geometry can be defined entirely in terms of those.

Our main goal in this work is to provide a theory of noncommutative spaces, in the sense of a generalization 
of Gelfand duality to noncommutative algebras, which incorporates from the outset the idea of localization by 
admitting a formulation in terms of sheaves.

It happens that the usual formulation of Gelfand duality obscures its relation to this ideas, since it deals with 
noncompact spaces by restricting the behavior ¨at infinity" of the admissible functions and, as we will argue 
latter, this removes the relation between a region and the others which may contain it.

Our proposal is that, in order to account for localization one must go beyond $C^*$ algebras, and consider the so 
called locally $C^*$ algebras, as defined by Inoue \cite{Ino} and studied by many other including \cite{APT}, \cite{Phi},
\cite{FRAG}. In this setting a trivial consequence of Gelfand duality is the following:

\begin{thm}\label{thm:GDR}
 Given a commutative $C^*$ algebra, there is a compact topological space, $X$, and a sheaf of 
 commutative locally $C^*$ algebras over it, such that the original algebra is isomorphic to the 
 algebra of global sections of this sheaf.
\end{thm}

Our main result is a extension of this to noncommutative locally $C^*$ algebras and compactly
generated spaces. 
To this end we turn to an interesting yet somewhat forgotten result  sometimes referred to as the 
sectional representation theorem. This is a consequence of some results proven by Dauns and Hofmann 
in the 60's (see \cite{DH} for a original reference or \cite{HOF} for a modern survey) and states that, 
for every $C^*$ algebra, there is a $C^*$ bundle, as will be defined defined in section \ref{sec:C*Alg}, 
such that the original algebra is isomorphic to some subalgebra of its algebra of continuous sections.

In Section \ref{sec:C*Alg} a corollary of the original sectional representation theorem, the one usually 
denoted by Dauns-Hofmann theorem in the modern $C^*$ algebra literature, is combined with other results
to provide a reinterpretation of  the original sectional representation theorem which, despite seeming 
weaker (the base space for the aforementioned bundle is then compact), hints on the generalization which 
is the main goal of this work.

Section \ref{sec:LocC*Alg} deals then with the generalization of the sectional representation theorem to 
the setting of locally $C^*$ algebras. 
To this end subsection \ref{subsec:PrimSpec} presents a convenient definition for the primitive spectrum of 
a locally $C^*$ algebra and provide some basic results about it. Subsection \ref{subsec:C*Fib} concludes the 
proof of one of our main theorems, which states that, for every locally $C^*$ algebra, there is a compactly 
generated space, the primitive spectrum of the center of its multiplier algebra, and a $C^*$ bundle over it, 
such that the original algebra is isomorphic to its algebra of continuous sections.

The connection between those bundles and sheaves of locally $C^*$ algebras is explored in section 
\ref{sec:Sheaf}.
We show that just as in the commutative case, a locally $C^*$ algebra always defines a sheaf of 
algebras, but to guarantee that this is actually a sheaf of locally $C^*$ algebras one must impose 
additional restrictions, such as dealing only with perfect locally $C^*$ algebras (cf. definition 
\ref{def:Perf}). 
Fortunately this case in broad enough to encompass most of the examples of interest such as $C^*$ algebras 
or $C^*$ bundles over locally compact spaces. We close the section with the tentative definition of a 
noncommutative space as the sheaf of locally $C^*$ algebras which is induced by a perfect locally $C^*$ 
algebra.

The paper is then concluded with some examples motivated by constructions from mathematical physics.
The first one is what the author propose to call "locally covariant quantum spacetime", a functor
between the category of Lorentzian manifolds of fixed dimension and noncommutative spaces, for which 
the stalks of the associated sheaf are isomorphic to the quantum spacetime algebra defined by Doplicher, 
Fredenhagen and Roberts in \cite{DFR}. Another class of examples is given by noncommutative spaces associated
to nets of $C^*$ algebras.
To define those we make use of a adaptation of a result proved by Ruzzi and and Vasseli in \cite{RV:NB} which 
shows that every algebraic quantum field theory, in the sense of Haag-Kastler axioms, induces a noncommutative 
space in our sense over the original spacetime.

\section*{Acknowledgments}

The authors are deeply indebted to K. Fredenhagen for all the support provided during Paulino's stay in 
Hamburg and the many invaluable suggestions and insights. Special thanks are also in due to P.L. Ribeiro 
and L.H.P. P\^egas for the many discussions that helped shape the ideas presented here.

\section{$\mathbf{C^*}$ Algebras and Compact Spaces}\label{sec:C*Alg}

A important object to be use in this work is the following

\begin{dfn}
 A \textbf{$\mathbf{C^*}$ Algebra Bundle}\/ or, more concisely, \textbf{$\mathbf{C^*}$ Bundle} over a topological 
 space~$X$ is a topological space $\mathcal{A}$ together with a surjective continuous and open 
 map $\, \xi: \mathcal{A} \longrightarrow X$, equipped with operations of fiberwise addition, scalar multi\-%
 plication, multiplication, involution and norm that turn each fiber $\, \mathcal{A}_x = \rho^{-1}(x) \,$ into a 
 $C^*$-algebra and are such that the corresponding maps
 \[
  \begin{array}{ccc}
   \mathcal{A} \times_X \mathcal{A} & \longrightarrow & \mathcal{A} \\[1mm]
               (a_1,a_2)             &   \longmapsto   &  a_1 + a_2
  \end{array} \quad , \quad
  \begin{array}{ccc}
   \mathbb{C} \times \mathcal{A} & \longrightarrow & \mathcal{A} \\[1mm]
             (\lambda,a)         &   \longmapsto   &  \lambda a
  \end{array}
 \]
 and
 \[
  \begin{array}{ccc}
   \mathcal{A} \times_X \mathcal{A} & \longrightarrow & \mathcal{A} \\[1mm]
               (a_1,a_2)             &   \longmapsto   &   a_1 a_2
  \end{array} \quad , \quad
  \begin{array}{ccc}
   \mathcal{A} & \longrightarrow & \mathcal{A} \\[1mm]
        a      &   \longmapsto   &     a^*
  \end{array}
 \]
 where $\, \mathcal{A} \times_X \mathcal{A} = \{ (a_1,a_2) \in \mathcal{A}\times \mathcal{A}~|~\xi(a_1) = \xi(a_2) \} 
 \,$ is the fiber product of~$\mathcal{A}$ with itself over~$X$, are all continuous.\footnote{Actually, it is sufficient 
 to require that scalar multiplication is continuous in the second variable, i.e., for each $\, \lambda \in \mathbb{C}$, 
 the map $\, \mathcal{A} \longrightarrow \mathcal{A}$, $a \longrightarrow \lambda a \,$ is continuous: this condition is 
 often easier to check in practice, but it already implies joint continuity~\cite[Proposition~C.17, p.~361]{Wil}.}
 Moreover, the function
 \[
  \begin{array}{ccc}
   \mathcal{A} & \longrightarrow & \mathbb{R} \\[1mm]
        a      &   \longmapsto   &    \|a\|
  \end{array}
 \]
 is supposed to be continuous or just upper semicontinuous, in which case one speaks of a \textbf{continuous} or 
 an \textbf{upper semicontinuous $\mathbf{C^*}$ bundle}, respectively, and to satisfy the following additional 
 continuity condition: any net $(a_i)_{i \in I}$ such that $\, \|a_i\| \to 0 \,$ and $\, \xi(a_i) \to x \,$ for 
 some $\, x \in X \,$ actually converges to $\, 0_x \in \mathcal{A}_x$. 
 Finally, we shall say that a $C^*$ bundle $\mathcal{A}$ is \textbf{unital} if all of its fibers $\mathcal{A}_x$ 
 are $C^*$ algebras with unit and, in addition, the unit section
\[
 \begin{array}{ccc}
  X & \longrightarrow & \mathcal{A} \\[1mm]
  x &   \longmapsto   &     1_x
 \end{array}
\]
 is continuous.
\end{dfn}

The study of this objects started with the early works of Fell, Dixmier and others for the case of continuous 
$C^*$-bundles,see for example \cite{FD} or \cite{DI}, and culminated with the works of Dauns and Hofmann in the 
mid 60's on the so called sectional representation theorems. In some of the of the earlier literature the term 
bundle used here is replaced by \emph{field}, and some times alternative equivalent definitions are used, but those 
have been abandoned in the modern literature in favor of the definitions and terminology presented 
here.%
\footnote{One criticism for the terminology adopted here is that therms like \emph{algebra bundle} are usually used 
as short hand for \emph{algebra fiber bundle}, 
were one usually requires the additional condition of local triviality. 
However the employment of the term \emph{bundle} as a generalization of \emph{fiber bundle} is so widespread in modern 
literature that the authors do not believe that significant confusion can arise.} 

Since the present work deal almost exclusively with this case, from now on, unless stated otherwise, all $C^*$ bundles 
are supposed to be upper semicontinous.

An important object associated to a $C^*$ bundle is the algebra of its continuous sections.
When $X$ is compact, the algebra $\Gamma(\mathcal{A})$ of all continuous sections of~$\mathcal{A}$, equipped 
with the usual pointwise defined operations of addition, scalar multiplication, multiplication and involution and 
with the usual sup norm,
\begin{equation} \label{eq:NSUP1}
 \|\varphi\|~=~\sup_{x \in X} \| \varphi(x) \|_x
 \qquad \mbox{for $\, \varphi \in \Gamma(\mathcal{A})$}~,
\end{equation}
is easily seen to be a $C^*$ algebra, and more than that: not only a $*$-algebra over the field of complex numbers 
but with the additional structure of a module over the $C^*$ algebra $C(X)$ of continuous functions on~$X$, subject 
to the compatibility conditions
\begin{equation} \label{eq:CSTMOD}
 f (\varphi_1 \varphi_2)~=~(f \varphi_1) \, \varphi_2~
 =~\varphi_1 \, (f \varphi_2) \quad , \quad
 (f \varphi)^*~=~\bar{f} \varphi^* \quad , \quad
 \| f \varphi \|~\leqslant~\| f \| \| \varphi \|~.
\end{equation}
When $X$ is locally compact but not compact, the situation is similar, but technically somewhat more complicated, because 
there are various choices to be made.
One of them consists in restricting to the algebra $\Gamma_0(\mathcal{A})$ of continuous sections of~$\mathcal{A}$ 
that vanish at infinity (in the usual sense that for each $\, \epsilon > 0$, there exists a compact subset $K$ of~$X$ such 
that $\, \|\varphi(x)\|_x < \epsilon \,$ whenever $\, x \notin K$), this is again a $C^*$ algebra and is even a module over 
the $C^*$ algebra $C_0(X)$ of continuous functions on~$X$ vanishing at infinity, subject to the same compatibility conditions
as before (see equation (\ref{eq:CSTMOD})), plus the condition of being nondegenerate, which states that the ideal generated 
by elements of the form $f \varphi$, with $\, f \in C_0(X) \,$ and $\, \varphi \in \Gamma_0(\mathcal{A})$, should be 
dense in~$\Gamma_0(\mathcal{A})$.
Note that the second case contains the first because when $X$ is compact, we can identify $C_0(X)$ with $C(X)$ and $\Gamma_0
(\mathcal{A})$ with $\Gamma(\mathcal{A})$ (since in that case the condition of vanishing at infinity is 
void and the nondegeneracy condition is automatically satisfied when the function algebra has a unit).
With this convention, we can describe the additional structure of~$\Gamma_0(\mathcal{A})$ as a module over $C_0(X)$ 
as being given by an embedding, in the sense of $C^*$ algebras, of $C_0(X)$ into the center $Z(M(\Gamma_0(\mathcal{A})))$ 
of the multiplier algebra $M(\Gamma_0(\mathcal{A}))$ of~$\Gamma_0(\mathcal{A})$.

This leads us to the definition of another important object in this work
\begin{dfn}
 Given a locally-compact topological space $X$, a $C_0(X)$ algebra is defined as a $C^*$ algebra $A$ equipped with a 
 homomorphism $\Phi: C_0(X) \rightarrow Z(M(A))$ which is non-degenerate, i.e. the closure of the ideal generated by 
 elements of the form $\Phi(f) a$ for $f\in C_0(X)$ and $a \in A$ is the whole algebra $A$. 
\end{dfn}

A comment that will be of interest to us latter on is that when $X$ is not compact we could replace $C_0(X)$ in the definition 
above by the $C^*$ algebra $C_b(X)$ of bounded continuous functions on~$X$, which has the advantage of being unital.
Obviously, $C_0(X)\subset C_b(X)$, and in fact $C_b(X)$ is just the multiplier algebra of~$C_0(X)$, which implies that for any 
$C^*$ algebra~$A$ with multiplier algebra $M(A)$, any nondegenerate $*$-homomorphism from $C_0(X)$ to~$M(A)$ extends uniquely 
to a $*$-homomorphism from $C_b(X)$ to~$M(A)$ \linebreak \cite[Corollary~2.51, p.~27]{RW}.
In particular, this means that \emph{any} $C_0(X)$ algebra is automatically also a $C_b(X)$ algebra.

As before in the compact case this structure is much simpler. When $X$ is compact the algebra $C_0(X) = C(X)$ is unital, so 
that the condition of nondegeneracy above reduces to the fact that the homomorphism $\Phi$ should be unital.

As it happens, \emph{every} $C^*$ algebra is a $C_0(X)$ algebra for a certain space $X$.
This is a consequence of the following result about $C^*$ algebras, \cite{MN},
\begin{thm}[Dauns-Hofmann]\label{thm:DH}
Given a $C^*$-algebra $A$ there is a canonical isomorphism $\Phi$ between  $C_b(\mathrm{Prim} A)$, the algebra of bounded functions 
over the primitive spectrum of $A$ and $Z(M(A))$ defined by the propriety that for every $P \in \mathrm{Prim} A$,
\[
  \Phi(f) a-f(P)a \in P
\]
for all elements $a \in A$ and $f \in C_b(\mathrm{Prim} A)$.
\end{thm}

Not surprisingly an analogous result to the Serre-Swan theorem for vector bundles relates both concepts defined above.

\begin{thm}\label{thm:SecRep}
 Given a $C^*$ algebra $A$ and a locally compact topological space $X$ the following statements are equivalent:
 \begin{itemize}
  \item $A$ is a $C_0(X)$ algebra.
  \item There is a $C^*$ bundle $\mathcal{A}$ over $X$ such that $A$ is isomorphic to $\Gamma_0(\mathcal{A})$.
  \item There is a continuous map $\chi: \mathrm{Prim} A \rightarrow X$
 \end{itemize}
\end{thm}
\noindent A certain feature of the proof of this theorem is that given a $C_0(X)$ algebra $A$
the fiber of the associated bundle over a point over a point $x\in X$, $\mathcal{A}_x$, is
defined by:
\[
 \mathcal{A}_x =  A\:/\: \overline{\lbrace f \in C_0(X)\: |\: f(x)=0 \rbrace \cdot A}.
\]
This fact will be of importance for our purposes latter on.
For a complete proof of this theorem the reader is referred to \cite[Theorem~C.26., p.~367]{Wil}.

Here we are faced with a technical difficulty  since the space $\mathrm{Prim}A$ may, in general
be very pathological, even non-Hausdorff. Because of this, before applying theorem \ref{thm:SecRep} 
we need a reformulation of the Dauns-Hofmann theorem. 

Given any topological space $X$ one may define its \emph{Stone-\v{C}ech compactification}. 
This is a compact Hausdorff space, denoted by $\beta X$, along with a canonical map\footnote{Here we implicitly use a generalization of the usual notion of compactification, it is important to note 
the distinction between $\beta X$, the Stone-\v{C}ech compactification of a space $X$, and $\beta (X) 
\subset \beta X$, the image of $X$ under the canonical map $\beta$. Had we required, as usual, the
space $X$ to be completely regular, the map beta would have turned out to be $\beta$  injective so 
that $X$ can then be identified with its image and such a distinction would not be necessary.}
$\beta: X \rightarrow \beta X$, defined by the universal propriety that for  every continuous 
mapping $f: X \rightarrow Y$ from $X$ to any compact Hausdorff topological space $Y$  there is a unique 
extension $\beta f$ to $\beta X$ such that $f = \beta f \circ \beta$. An interesting fact is that for any 
topological space $X$ the Stone-\v{C}ech compactification can be defined by $\beta X = \mathrm{Prim} C_b(X)$, 
so that we have 
\[
 C(\beta X) \approx C_b(X).
\]
This shows that every $C_0(X)$ algebra is a $C(\beta X)$ algebra, in particular the isomorphism from the 
Dauns-Hofmann theorem implies that
\[
 \beta \mathrm{Prim} A \approx \mathrm{Prim} Z(M(A)).
\] 
and so that every $C^*$ algebra is a $C(\beta \mathrm{Prim} A)$ algebra.
The compact Hausdorff topological space $\beta \mathrm{Prim} A$ is denoted by $\mathrm{pt} A$
and call it the space of points of the algebra $A$. Combining this remarks with theorem 
\ref{thm:SecRep} we get the following reformulation of the sectional representation theorem,
proved by Dauns and Hofmann in \cite{DH}.

\begin{thm}\label{thm:NCGT}
 Every $C^*$ algebra $A$ is isomorphic to the algebra of continuous sections
 of a $C^*$ bundle $\mathcal{A}\rightarrow \mathrm{pt} A$ over its space of
 points. Moreover, for $P\in \pt A$, the associated fiber $\mathcal{A}_P$
 can be written as
 \[
  \mathcal{A}_{P} = A / (P \cdot A),
 \]
 where $P \cdot A$ denotes the closure of the ideal in $A$ generated by elements of the form 
 $fa$ where $f \in P \subset Z(M(A))$ and $a \in A$.
\end{thm}

This theorem reveals some curious facts when $A$ is a commutative $C^*$ algebra.
If the algebra $A$ is unital, $\mathrm{Prim}A = \mathrm{pt}A$ is compact and one can easily 
check that $\mathrm{pt}A\times \mathbb{C}$ is precisely the $C^*$ bundle given by our theorem, 
so that one gets
\[
  A \approx \Gamma(\mathrm{pt}A \times \mathbb{C}) \approx C(\mathrm{pt}A),
\] 
recovering precisely the original commutative Gelfand theorem.

When $A$ is nonunital one need to use the canonical map $\beta$  between the primitive spectrum $\mathrm{Prim}A$ 
and the space of points $\mathrm{pt}A = \beta \mathrm{Prim}A$. Then
\[
  C_0 (\mathrm{Prim}A) \approx
\lbrace f \in C(\mathrm{pt}A) \: | \: \forall x \not\in \beta(\mathrm{Prim}A) \:\: f(x) = 0\rbrace\]
We define a sub-bundle $\mathcal{A} \subset \mathrm{pt}A \times 
\mathbb{C}$ by
\begin{equation*}
 \mathcal{A}_x = \left\lbrace \begin{array}{cc}
                     \mathbb{C} & x \in \beta(\mathrm{Prim}A)  \\
                     \lbrace\emptyset\rbrace  & x\not\in \beta(\mathrm{Prim}A)
                    \end{array} \right.
\end{equation*}
Then $C_0 (\mathrm{Prim}A) \approx \Gamma(\mathcal{A}) \approx A$ and this
gives us back the original Gelfand theorem for nonunital commutative $C^*$-algebras.

The previous remarks point to an interesting feature of $C^*$ algebras that are usualy 
overlooked in the formulation of the original Gelfand Theorem. Due to the last identification 
above we can not distinguish between the algebra of functions vanishing at infinity over a 
noncompact space and some closed subalgebra of the functions over its Stone-\v{C}ech 
compactification. 

One may then take this as a indication that $C^*$ algebras are ill suited to deal with noncompact
spaces; by restricting the behavior of the sections at infinity one loses the information
about the noncompactness.

\section{Noncompact Spaces and Locally $C^*$ Algebras}\label{sec:LocC*Alg}

It is a well know fact that a locally compact space $X$ is noncompact if and only if 
there is a continuous unbounded function $f\in C(X)$.
This hints that one should consider a mathematical structure that generalizes the concept 
of $C^*$-algebra as to be able to deal with possibly unbounded functions on noncompact 
spaces. 
Fortunately, this structure already exists: it has been introduced in the early 1970's 
under the name ``locally $C^*$ algebra''~\cite{Ino} and further investigated, partly 
under other names such as ``pro-$C^*$-algebra'' by various authors; see, e.g., 
\cite{BK}, \cite{FRAG}, \cite{Phi}:
\begin{dfn}
 A\/ \textbf{locally $C^*$-algebra} $A$ is a $*$-algebra equipped with
 a locally convex topology which is Hausdorff, complete and generated
 by a family of $C^*$-seminorms.
\end{dfn}
\noindent
For the sake of definiteness, we recall here that a $C^*$-seminorm on
a $*$-algebra~$A$ is a seminorm $s$ in the usual sense (i.e., on~$A$
as a vector space) which satisfies the additional requirements for a
$C^*$-norm except for definiteness, namely,
\begin{equation} \label{eq:CSTSN1}
 s(a_1 a_2)~\leqslant~s(a_1) \, s(a_2) \quad , \quad s(a^*)~=~s(a)
 \quad , \quad s(a) \geqslant 0
\end{equation}
and
\begin{equation} \label{eq:CSTSN2}
 s(a^* a)~=~s(a)^2~.
\end{equation}
Note that for a given locally $C^*$-algebra $A$, there may of course be
many different families of $C^*$-seminorms that generate its topology,
and by taking maximums over finite sets, we can always work with
families that are saturated.%
\footnote{A family $(s_i)_{i \in I}$ of seminorms on a vector space is
said to be saturated if for each finite subset $\{i_1,\ldots,i_p\}$
of~$I$, there exists some $\, i \in I \,$ such that $\, s_{i_k}(x)
\leqslant s_i(x) \,$ for all~$x$ and $\, 1 \leqslant k \leqslant p$.}

Given a locally $C^*$ algebra $A$ we denote by $S(A)$ the directed set of all continuous
$C^*$ seminorms in it's locally convex topology, equipped with the obvious order relation.
For each seminorm $s \in S(A)$, the kernel of $s$ is a closed $*$-ideal of~$A$, 
so that we can define a $C^*$-algebra $A_s$ as the completion of the quotient 
of~$A$ by the kernel of $s$, which is a normed $*$-algebra with respect to the 
$C^*$-norm $\|.\|_s$ induced on it by~$s$. (In fact, it turns out \emph{a posteriori} 
that this quotient is already complete, so $\, A_s = A\,/\ker s$.)
Moreover, we can show that $A$ is the \emph{inverse limit},
also called the \emph{projective limit}, of the family of $C^*$ algebras 
$(A_s)_{s \in S(A)}$. 
The family of $C^*$ algebras $(A_s)_{s \in S(A)}$ is called the \emph{Michael-Arens 
decomposition} of the algebra $A$.

The basic examples of this kind of structure are of course provided by continuous functions and, 
more generally, continuous sections of $C^*$-bundles over compactly generated spaces.
Indeed, given any such space~$X$, the $*$-algebra $C(X)$ of continuous functions on~$X$ is a locally 
$C^*$ algebra with respect to the compact-open topology, which is defined by uniform convergence over 
compact sets, or equivalently by the family of $C^*$-seminorms $(\|.\|_K)_{K \subset X, K\;
\mathrm{compact}}$ given by
\begin{equation} \label{eq:NSUP2}
 \|f\|_K~=~\sup_{x \in K} |f(x)|
 \qquad \mbox{for $\, f \in C(X)$}~.
\end{equation}
More generally, given any $C^*$-bundle $\mathcal{A}$ over~$X$, the $*$-algebra $\Gamma(\mathcal{A})$ 
of continuous sections of~$\mathcal{A}$ is a locally $C^*$-algebra with respect to its natural topology, 
which is again that of uniform convergence on compact subsets, defined by the family of $C^*$-seminorms 
$(\|.\|_K)_{K \subset X, K\;\mathrm{compact}}$.

It happens that, unfortunately, this is not the most general situation.

\begin{dfn}
 Given a topological space $X$ a \textbf{distinguished family of compact sets} 
 is a family $F$ of compact sets in $X$ which satisfies
 \begin{itemize}
 
  \item Every singleton of $X$ is in $F$.
 
  \item Every compact subset of a compact in $F$ is in $F$.
 
  \item The union of two elements in $F$ is again in $F$.
 
  \item The space $X$ is the direct limit of the family $F$
       ordered by inclusion, i.e. $\bigcup F = X$ and
       \[
        C \;\text{is closed in}\; X \Longleftrightarrow C \cap K \; \text{is closed in} \;
        K \; \text{for every} \; K\in F.
       \]
 \end{itemize}
 A \textbf{Compactly Generated Space} is a topological space which admits a distinguished
 family of compact sets.
\end{dfn}
As was shown in \cite{Phi}, for any such a family, we can define a new locally $C^*$ topology on the space 
of continuous functions over $X$ by uniform convergence over the compacts in $F$, i.e. given by the 
$C^*$-seminorms $(\|.\|_K)_{K \in F}$. It can be shown that for two different distinguished families over 
a functionally Hausdorff\footnote{i.e. such that points can be separated by continuous functions.} space one 
can define a net of functions that converges to zero only in one of the topologies, so that, if the space $X$ 
admits more than one family of distinguished compact sets, one may have more than one locally $C^*$ topology in 
the algebra $C(X)$.

This example can then be generalized to the continuous sections of a $C^*$ bundle.
Given a $C^*$ bundle $\mathcal{A} \rightarrow X$ over a compactly generated space $X$ with a distinguished 
family of compact sets $F$ one may define a locally $C^*$ topology in the algebra of continuous sections 
$\Gamma(\mathcal{A})$ using the seminorms:
\begin{equation} \label{eq:NSUP4}
 \|\varphi\|_K~=~\sup_{x \in K} \| \varphi(x) \|_x
 \qquad \mbox{for $\, \varphi \in \Gamma(\mathcal{A})$}~.
\end{equation}
for $K\in F$.

\subsection{The Primitive Spectrum of a Locally $C^*$ Algebra}\label{subsec:PrimSpec}

Our main aim is now to show how theorem \ref{thm:NCGT} can be extended to locally $C^*$ algebras. 
To do so we need first to extend our knowledge about the primitive spectrum from the setting of $C^*$ algebras 
to that of locally $C^*$ algebras.
Since there seems to be no systematic account of this in the literature, we present here a brief collection of 
definitions and results that remain valid in the context of locally $C^*$ algebras. 
Most of what is done here is largely based on the presentation of the primitive spectrum in \cite[Sec. A.2]{RW}

\begin{dfn}
 Let $A$ be a locally $C^*$ algebra. An ideal $P \subset A$ is said to be a \textbf{primitive ideal} if
 there is a irreducible\footnote{As is the case for $C^*$ algebras a representation of a locally $C^*$
 is topologically irreducible if and only if it is algebraically irreducible, \cite{BK}, and so we make 
 no distinction between both concepts.} representation $\pi$ of $A$ such that $P = \ker \pi$. The set of 
 all primitive ideals in $A$ is called the \textbf{primitive spectrum} and denoted by $\mathrm{Prim} A$.
\end{dfn}

Before introducing a topology on the primitive spectrum we prove a basic lemma about primitive ideals.

\begin{lem}\label{lem:intprime}
 Let $A$ be a locally $C^*$ algebra. 
 \begin{itemize}
  \item Every closed ideal in $A$ is the intersection of the primitive ideals containing it.
  
  \item if $P$  is a primitive ideal in $A$ and $I$ and $J$ are ideals such that $I \cap J
  \subset P$ then either $I \subset P$  or $J \subset P$
 \end{itemize}
\end{lem}

\begin{proof}
 To prove the first statement we show that if $I$ is a closed ideal in $A$ and $a \notin I$ then there is
 a primitive ideal $P \in \Prim A$ such that $I \subset P$ and $a \notin P$. We first note that as was proved 
 in \cite[Theorem 11.7, p. 140]{FRAG} the algebra $A/I$ admits a topology generated by $C^*$-seminorms and so 
 has a locally $C^*$ completion (see remarks in \cite[pp. 14, 102]{FRAG}). Let $A_I$ be this completion and 
 $q_I: A \rightarrow A_I$ the quotient mapping. 
 Then $A_I \ni a+I \neq 0$ and so that there is a irreducible representation $\pi^\prime$ of $A_I$ such 
 that $\pi^\prime (a + I) \neq 0$ and thus $\pi = \pi^\prime \circ q_I$ is a irreducible representation 
 of $A$ such that $a \notin P = \ker \pi$. 
 
 To prove the second claim let $\pi: A \rightarrow \mathcal{B}(\mathcal{H})$ be an irreducible representation
 such that $P  = \ker \pi$. If $I \slashed\subset P$ then $\pi (I) \neq 0$ and so $\mathcal{V} = 
 \overline{\pi(I)\mathcal{H}}$ is also nonzero. Since $I$ is an ideal $\mathcal{V}$ is invariant and, since 
 $\pi$ is irreducible this implies that $\mathcal{V} = \mathcal{H}$, so that
 \[
  \pi(J)\mathcal{H} = \pi(J)\left( \pi(I) \mathcal{H}  \right) \subset \pi(I\cap J) \mathcal{H} \subset  \pi(P) \mathcal{H} = 0 
 \]
 proving that $J \subset \ker \pi = P$.
\end{proof}

We define then a topology on $\Prim A$ by the following

\begin{dfn}
 Given a locally $C^*$ algebra $A$ and a set $F \subset \Prim A$ we define the closure  $\overline{F}$ 
 of $F$ by
 \[
  \overline{F} = \lbrace P \in \Prim A, \: \bigcap_{I \in F} I \subset P \rbrace .
 \]
\end{dfn}

It is an easy exercise to show that this prescription defines a topology in $\Prim  A$. As
in the case of $C^*$ algebras we call this the \emph{hull-kernel topology}. The next proposition
gives an alternative description of the closed sets in $\Prim A$ and justifies the name given to 
this topology.

\begin{prp}
 Let  $A$ be a locally $C^*$ algebra. The prescription $F \mapsto k(F) = \bigcap F$ is a bijection
 between the closed sets in the hull-kernel topology and the closed ideals of $A$. Its inverse is 
 given by
 \[
  h(I) = \lbrace P \in \Prim A, \: I \subset P \rbrace
 \]
we call $k(F)$ the \textbf{kernel} of $F$ and $h(I)$ the \textbf{hull} of $I$.
\end{prp}

\noindent We omit here the trivial proof of this statement.

The following theorem relates the primitive spectrum of the original locally $C^*$ algebra, and those of
the algebras in its Michael-Arens decomposition.

\begin{thm}\label{thm:prmdec}
 Let $A$ be a locally $C^*$. Then for each seminorm $s \in S(A)$ there is an isomorphism between $\Prim A_s$ 
 and the closed subspace $h(\ker s)$ of $\Prim A$ when equipped with the hull-kernel topology.
\end{thm}

\begin{proof}
 Let $q_s: A \rightarrow A_s= A/\ker s$ we denote the quotient map, then if $Q \in \Prim A$ is such that $Q \supset \ker s$
 and $Q = \ker \pi$ for some irreducible representation $\pi$, then $\pi$ defines a irreducible representation $\pi_s$ of 
 $A_s$ and $\ker \pi_s = q_s^{-1}(Q)$. If $\pi_s$ is a irreducible representation of $A_s$ then $\pi = \pi_s \circ q_s$ is 
 a irreducible representation of $A$, $\ker \pi \supset \ker s$ and $\ker \pi_s = \ker \pi / \ker s$.  
 It is clear then that $Q \mapsto q_s^{-1}(Q)$ is a bijection between $\mathrm{Prim} A_s$ and the closed set $h(\ker s) 
 \subset \Prim A$ with inverse $P \mapsto P /\ker s$. 
 
 To show that this bijection is actually an homeomorphism we note that, since $q_s^{-1} (\mathrm{Prim} A_s) = h(\ker s)$
 is closed, its closed sets are precisely the ones already closed in $\Prim A$, that is, of the form $h(I)$ for some ideal 
 $I \supset \ker s$. 
 But the closed sets in $\mathrm{Prim} A_s$ are precisely of the form $h_s(J)= \lbrace Q \in \mathrm{Prim} A_s,\: J 
 \subset Q \rbrace$ for some ideal $J$ of $A_s$, so that $\rho_s^{-1}$
 maps $h_s(J)$ to $h(q_s^{-1}(J))$ and its inverse maps $h(I)$ to $h_s(I /\ker s)$. 
\end{proof}
 
First of all, since the pull-back of a irreducible representation by a surjective map is again a irreducible representation $\Prim$ 
is a contravariant functor when restricted to the category of $C^*$ algebras with surjective morphisms, and so it is clear that 
$(\Prim A_s)_{s \in S(A)}$ is a inductive system of topological spaces. 
Moreover, since the sets $h (\ker s)$ cover $\Prim A$, one can use theorem \ref{thm:prmdec} to identify $\Prim A$ with the set-theoretical 
direct limit of the family $(\Prim A_s)_{s \in S(A)}$.

\begin{dfn}
 We define the \textbf{direct limit topology} on $\Prim A$ as the topology given by the identification between $\Prim A$
 and $\varinjlim \Prim A_s$ discussed above. That is, a set in $\Prim A$ is closed in the direct limit topology if, and only
 if, its intersection with any subspace of the form $h (\ker s)$ is closed.
\end{dfn}

A first remark about the definition above is that, in the case of $C^*$ algebras the direct limit topology is precisely the
hull-kernel topology, since the family $(A_s)_{s \in S(A)}$ has a "upper bound", $A$ itself. As we shall see latter, there
is at least one more special class of locally $C^*$ algebras, that of perfect locally $C^*$ algebras, in which both topologies
coincide. 

We are now ready to prove a version of the Gelfand theorem for locally $C^*$ algebras.

\begin{thm}\label{thm:GTLC*}
 Let $A$ be a commutative unital locally $C^*$ algebra. Then 
 \begin{itemize}
  \item $\Prim A$, equipped with the direct limit topology, is a compactly generated, functionally Hausdorff space and 
  $(\Prim A_s)_{s \in S(A)}$ is a distinguished family of compact sets.
  \item The algebra $A$ is isomorphic to $C(\Prim A)$ when equipped with the topology associated to $(\Prim A_s)_{s \in S(A)}$.
 \end{itemize}
 Moreover, the functors
 \[
   X \mapsto C(X), 
 \]
 and
 \[
   A \mapsto \Prim A,
 \]
 provide a duality between the categories of commutative unital locally $C^*$ algebras and compactly generated, functionally 
 Hausdorff topological spaces equipped with a distinguished family of compact sets\footnote{Here the morphisms are continuous 
 maps such that the image of each set in the distinguished family in one space is a member of the distinguished family on the 
 other space.}.
\end{thm}

\begin{proof}
 First, when $A$ is commutative and unital, each $\Prim A_s$ is a compact Hausdorff topological space, and so $\Prim A$ is a
 compactly generated space. Besides that, By the very definition of the direct limit topology, $(\Prim A_s)_{s \in S(A)}$ is 
 a distinguished family of compact sets. From now on, unless otherwise stated, $\Prim A$ is to be equipped with the direct
 limit topology and the distinguished family of compact sets $(\Prim A_s)_{s \in S(A)}$.
 
 By the commutative Gelfand theorem $A_s$ is isomorphic to $C(\Prim A_s)$, and thus
 \[
  A = \varprojlim A_s \approx \varprojlim C(\Prim A_s) \approx C(\varinjlim \Prim A_s) = C(\Prim A) 
 \]
 Where the first isomorphism is the one given by the universal propriety of the inverse limit applied to the aforementioned
 Gelfand isomorphisms and the second one is consequence of the universal proprieties defining the limits involved.
 
 Suppose now that $P_1$ and $P_2$ are points in $\Prim A$ which can not be separated by a continuous function.
 We know that there are irreducible representations $\pi_j$, $j=1,\:2$, of $A/P_j$ such that, denoting by 
 $q_j$ the quotient mapping, $P_j = \ker \pi_j \circ q_j$.
 Now let $\Phi$ be isomorphism between $A$ and $C(\Prim A)$ and $\Phi_j$ be the isomorphism between $A / P_j$ and 
 $C(\Prim A / P_j)$\footnote{We note that, since $P_j$ is a primitive ideal, the algebras $A / P_j$ are in fact $C^*$
 algebras, and so the referred isomorphisms are precisely those given by the respective Gelfand maps.}. 
 Since $\Phi$ is defined by the universal propriety as remarked above we have
 \[
  \pi_1 (q_1(a)) = \Phi_1(q_1 (a))(P_1) = \Phi(a)(P_1)  = \Phi(a)(P_2) = \Phi_2(q_2 (a))(P_2) = \pi_2(q_2(a))
 \]
 for all $a \in A$ so that $P_1 = P_2$ and thus $\Prim A$ is functionally Hausdorff.
 
 Due to the very definition of morphism in the category of functionally Hausdorff compactly generated spaces
 it is easy to see that the identification $X \mapsto C(X)$ is a contravariant functor between that appropriate
 categories, and so is $A\mapsto \Prim A$, since the pull-back of a irreducible representation by a unital morphism 
 between commutative unital algebras is again a irreducible representation. Furthermore, it is a trivial exercise to 
 verify that these functors constitute a duality between both categories. 
\end{proof}

This result was originally proved by Phillips in \cite{Phi} using the space of characters over the locally $C^*$ algebra
instead of the primitive spectrum as was done here. A more recent reformulation of Phillip's result was given by El Harti 
and Luk\'acs in \cite{Harti}. As shall be seen in the next section, this approach has the advantage of relating easily to 
the sectional representation theorem.

\subsection{The $C^*$-Bundle of a Locally $C^*$ Algebra}\label{subsec:C*Fib}

In \cite[Theorem 11.5]{Phi} it was shown that, for any locally $C^*$ algebra $A$, the multiplier algebra $M(A)$ 
is the projective limit of the family $(M(A_s))_{s \in S(A)}$ induced by the Michael-Arens decomposition of $A$. 
Inspired by our results for $C^*$ algebras we propose  then the following definition
\begin{dfn}
 Given a locally $C^*$ algebra $A$, with Michael-Arens decomposition $(A_s)_{s \in S(A)}$, we define the \textbf{
 space of points} of $A$, denoted by $\pt A$, as the functionally Hausdorff comp generated space
 \[
  \pt A = \Prim Z(M(A)) = \varinjlim \Prim Z(M(A_s))
 \]
 equipped with the distinguished family of compact sets given by $\pt A_s = \Prim Z(M(A_s))$. 
\end{dfn}

By theorem \ref{thm:prmdec} we know that there is an isomorphism between $C(\mathrm{pt} A)$ and $Z(M(A))$ which is
induced by the isomorphisms between $C(\pt A_s)$ and $Z(M(A_s))$. 

We formulate then our main theorem

\begin{thm}\label{thm:NCGT2}
 Let $A$ be a locally $C^*$ algebra. Then there is a $C^*$ bundle $\mathcal{A} \rightarrow \pt A$ such that
 \[
  A \approx \Gamma(\mathcal{A})
 \]
\end{thm}

\begin{proof}
Let, for each $s \in S(A)$, $\mathcal{A}^s$ be the bundle given by theorem \ref{thm:NCGT} for the algebra $A_s$ in
the Michael-Arens decomposition of $A$.

Since the algebras $A_s$ are already complete, we know the quotient maps $q_s: A \rightarrow A_s$ are all surjective,
and so, for every pair of seminorms $s$ and $r$ in $S(A)$, such that $\ker r \supset \ker s$, we get another
surjective mapping $ q_{rs}: A_s \rightarrow A_r $.

Given $P \in \pt A$ such that $P \supset \ker r \supset \ker s$, denote by $P_s =  P/\ker s \in \pt A_s$ and $P_r = P / \ker r \in \pt A_r$. 
We have
\begin{eqnarray*}
 A_s / (P_s \cdot A_s) & = & (A / \ker s) / ((P \cdot A) / \ker s) \\ 
                       & \approx & A/(P\cdot A) \\
                       & \approx & (A / \ker r) / ((P \cdot A) / \ker r) \\
                       & = & A_r /(P_r \cdot A_r)
\end{eqnarray*}
This shows that the fibers in the bundles $\mathcal{A}^s$ associated to points that are identified in the inductive
system $(\pt A_s)_{s \in S(A)}$ are isomorphic.
We construct then the set-theoretical commutative diagrams
\[
\begin{CD}
\mathcal{A}^r @>{\xi_r}>> \pt A_r\\
@V\rho_{rs}VV @VV\pt\, q_{rs}V\\
\mathcal{A}^s @>>\xi_s> \pt A_s
\end{CD}
\]
Now we know from the proof in \cite[Theorem~C.25, p.~364]{Wil} that a basis for the topology in $\mathcal{A}^s$ is given by 
sets of the from
\[
 \lbrace b \in \xi_s^{-1}(U),\: \|b - q_{\pi_s(b)}(a)\| \leqslant \epsilon \rbrace
\]
For a fixed $a \in A_s$ and $\epsilon > 0$. Now, due to the commutativity of the diagram above we know that the pre-image 
of this sets under $\rho_{sr}$ are of the form
\[
 \lbrace b \in \xi_r^{-1}(\pt\, q_{rs}^{-1} (U)),\: \|b - q_{\pi_r(b)}(q_{rs}(a))\| \leqslant \epsilon \rbrace
\]
Now, since the maps $q_{rs}$ are surjections, we conclude that $\rho_{sr}$ are not only continuous, but homeomorphism into
their images, and thus the family $(\mathcal{A}^s)_{s \in S(A)}$ constitutes a inductive family of $C^*$-bundles. 
We can then define the direct limit
\[
 \mathcal{A} = \varinjlim \mathcal{A}^s
\]
and since the maps $\rho_{sr}$ are homeomorphisms into their images, this is a $C^*$ bundle over 
$\pt A = \varinjlim \pt A_s$.

We know from theorem \ref{thm:NCGT} that $A_s \approx \Gamma(\mathcal{A}^s)$, so that
by universal proprieties of the limits involved
\[
 A = \varprojlim A_s \approx \varprojlim \Gamma(\mathcal{A}^s) \approx \Gamma(\varinjlim 
 (\mathcal{A}^s)) = \Gamma(\mathcal{A})
\]

\end{proof}

\section{Sheaves and Perfect Locally $C^*$ Algebras}\label{sec:Sheaf}

The equivalence between sheaves and bundles is recurring subject in the literature, it was explored in the most 
different contexts, from algebraic geometry to functional analysis. The authors have clearly no hope of providing a 
comprehensive account of this subject, so instead this section focus on the specific case of locally $C^*$ algebras.

For open sets $U$ in the direct limit topology on the space of points of $A$, the the prescription $U \mapsto \Gamma(
\mathcal{A},U)$ defines a sheaf of algebras over $\pt A$. The question is then when can we equip those algebras $\Gamma(
\mathcal{A},U)$ with a locally $C^*$ topology defined by $A$ so as to obtain a sheaf of locally $C^*$ algebras associated 
to the original one.

To endow $\Gamma(\mathcal{A}, U)$  with a locally $C^*$ topology one must require that the open set $U$ is 
itself compactly generated. Unfortunately, in general, there is no way to guarantee this, and one must require that 
the open set is \emph{regular}, i.e. contains a closed neighborhood of each of its points, to do so.
To require that every open subset of $\pt A$ is regular one must impose additional restrictions, for example,
requiring that the space of points of the algebra is not only compactly generated, but locally compact. 
Fortunately, this requirement have a interesting formulation in terms intrinsic to a given locally $C^*$ algebra. 

To give this formulation one needs the following definitions.  

\begin{dfn}\label{def:Spd}
 Given a locally $C^*$ algebra $A$ and a seminorm $s \in S(A)$, an element $a\in A$ is said to be 
 \textbf{supported in} $s$ if
 \[
  a \cdot \ker s = 0
 \]

The two-sided ideal of all the elements supported in a seminorm $s$ is denoted by $\Sppd s$.
\end{dfn}

\noindent Here our notation is clearly inspired by the commutative case, where an element is supported in $s$ if and
only if the associated element in $C(\pt A)$ is supported, in the usual sense, in the compact $\pt A_s \subset pt A$. 
With this definition in mind we define a special class of locally $C^*$ algebras.

\begin{dfn}\label{def:Perf}
 A locally $C^*$ algebra, $A$, is said to be \textbf{perfect} if \[\overline{\sum_{s\in S(A)} \Sppd s} = A\]
 that is, if the ideal generated by all the elements supported on some seminorm $s$ in $S(A)$  is dense in $A$.
\end{dfn}

Clearly, every $C^*$ algebra is perfect. Moreover if the center of the multiplier algebra of a given locally $C^*$ 
algebra is perfect then so is the original algebra.

It was proven by Apostol in \cite[p. 36 Theorem 4.1]{APT} that a unital commutative locally $C^*$ algebra is perfect 
if, and only if, it is isomorphic the algebra of continuous functions over a locally compact space\footnote{An interesting 
restatement of this result, which was unknown to the author, is that a topological space is locally compact if and only if 
every function over it can be approximated, in the compact-open topology, by functions with compact support.} equipped
with the compact-open topology.

Combining these remarks we obtain the generalization of theorem \ref{thm:GDR} to noncommutative algebras.

\begin{thm}\label{thm:Sheaves}
 Let $A$ be a locally $C^*$ algebra such that the center of its multiplier algebra is perfect. Then the space 
 of points $\pt A$ is a locally compact topological space.
 In this case, for every open set $U \subset \pt A$, the algebra $\Gamma(\mathcal{A}, U)$ admits a locally 
 $C^*$ topology, and so the functor $U \mapsto \Gamma(\mathcal{A}, U)$ defines a sheaf of locally $C^*$ algebras.
\end{thm}

Upon restriction to the case of a normable topology (i.e. for $C^*$ algebras) we get the following corollary, 
showing how our notation reflects perfectly the underling structure
\begin{cor}
 Every $C^*$ algebra is isomorphic to the algebra of global sections of a sheaf of locally $C^*$ algebras over
 its space of points.
\end{cor}

Perfect locally $C^*$ algebras are also of interest because of the following lemma, which shows that, in this case 
the topology in $\Prim A$ has a intrinsic definition in terms of the algebraic structure of $A$.

\begin{lem}
 For a perfect locally $C^*$ algebra the direct limit topology and the hull-kernel topology coincide.
\end{lem}
\begin{proof}
 For any locally $C^*$ algebra, the intersection of any closed set $C \in \Prim A$ in the hull-kernel topology
 with a set of the form $h(\ker s)$ is closed and, therefore, $C$ is also closed in the direct limit topology.
 Now let $C \subset \Prim A$ be a closed set in the direct limit topology. This means that, for every $s\in S(A)$,
 $C \cap h(\ker s)$ is closed, i.e.
 \[
   C \cap h(\ker s) = \lbrace Q \in \Prim A,\: Q \supset k(C \cap h(\ker s))\rbrace
 \]
 We will show that if $P$ is an primitive ideal on the closure of $C$ in the hull-kernel topology, then $P \in C$, so
 that $C$ is closed in the hull-kernel topology also. To this end we notice that, by hypothesis
 \[
 	P \supset \bigcap C  = \bigcap (C\cap h(\ker s)) \cap \bigcap( C\cap \mathcal{O}(\ker s))
 \]
 where $\bigcap C = \bigcap_{Q \in C} Q$ and $\mathcal{O}(I) = \lbrace Q \in \Prim A,\: Q \slashed\supset  I \rbrace$.

 By lemma \ref{lem:intprime}, $P$ is prime because it is primitive, and so the relation above implies either 
 $P \supset \bigcap (C\cap h(\ker s))$ 
 or $P \supset \bigcap (C\cap \mathcal{O}(\ker s))$. Now, since $P$ is a proper closed subset of $A$, which 
 is perfect by hypothesis, we know that there is an $s_0 \in S(A)$ such that $\Sppd s_0  \mathop \backslash  P 
 \neq \emptyset$.

 For every $Q \in \mathcal{O}(\ker s_0)$ we have $Q \supset Sppd s_0 \cap \ker s = \lbrace 0 \rbrace$ and so 
 $Q \supset \Sppd s_0$ since $Q$ is also primitive, and so prime. 
 Then, either $C \cap \mathcal{O}(\ker s_0) = \emptyset$, which implies that $C \subset h(ker s_0)$, and so is 
 closed in the hull-kernel topology by hypothesis or $\bigcap( C\cap \mathcal{O}(\ker s_0)) \supset \Sppd s_0$.
 Since $\Sppd s  \mathop \backslash  P \neq \emptyset$, this implies $P \slashed\supset  \bigcap( C\cap \mathcal{O}(\ker s))$,
 so that $P \supset \bigcap (C\cap h(\ker s))$, and thus $P \in C$

\end{proof}

%
%
%
%
%
%

Finally  we are able to present our definition of noncommutative space.

\begin{dfn*}
 A \textbf{Noncommutative Space} is a sheaf of locally $C^*$ algebras whose multiplier algebras 
 have perfect centers.
\end{dfn*}

\section{Examples of Noncommutative Spaces}\label{sec:Exp}

\subsection{The DFR Algebra of a Poisson Vector Bundle}

In \cite{FP} the authors showed how to construct a $C^*$ bundle which extends, in a certain sense, 
the construction of a $C^*$ algebra from commutation relations encoded by a Poisson tensor on a vector 
spaces to vector bundles.

Given a Poisson vector space $(V, \sigma)$, i.e., a real vector space, $V$, equipped with a fixed bi-vector, 
$\sigma$, one can construct Fr\'echet $*$-algebra by  equipping the space of Schwartz functions over $V$ with 
the usual involution and the product defined by
\begin{equation*} \label{eq:WMOY2}
 (f \star_\sigma g)(x)~=~\int_{V^*} d\xi~e^{i\langle\xi,x\rangle}
 \int_{V^*} d\eta~\check{f}(\eta) \, \check{g}(\xi-\eta) \;
 e^{\frac{i}{2} \sigma(\xi,\eta)}~,
\end{equation*}
this algebra is denoted by $\mathscr{S}_\sigma$, and called \emph{Heisenberg-Schwartz algebra}\footnote{
As remarked in \cite{FP} the is just the algebra obtained by Rieffel's deformation if one
sets $J = - \pi \sigma^\sharp$}.

In \cite{FP} the authors showed that this algebra has a unique $C^*$ completion, the \emph{Heisenberg $C^*$ 
algebra}, denoted by $\mathscr{E}_\sigma$ and proved that the Heisenberg-Schwartz algebra is spectraly invariant 
over this completion. The usual algebra of commutation relations, generated by a representation of the induced
Heisenberg group, can be then recovered as a subalgebra of the multiplier algebra of our Heisenberg $C^*$-algebra.

Given a Poisson vector bundle $(E, \sigma)$ over a manifold $M$, i.e. a vector bundle $\rho: E\rightarrow M$ and a 
bivector field $\sigma$ over $E$, one defines then an associated vector bundle, $\mathscr{S}(E)$ whose fibers are 
the spaces of Schwartz functions over the original fibers. On the space of sections of this bundle we then define 
a product by:
\begin{equation*} \label{eq:WMOY3}
 (f \star_\sigma g)(m)(e)~=~\int_{E_m^*} d\xi~e^{i\langle\xi,e\rangle}
 \int_{E_m^*} d\eta~\check{f}(m)(\eta) \, \check{g}(m)(\xi-\eta) \;
 e^{\frac{i}{2} \sigma_m(\xi,\eta)}~,
\end{equation*}
where $e\in E_m$ and $\check{f}(m)$ and $\check{g}(m)$ denote the Fourier transforms of $f(m)$ and $g(m)$ in 
the usual sense.

The authors showed then that the subalgebra of sections with compact support of this algebra admits a $C^*$ norm 
and so a $C^*$ completion, and by using the sectional representation theorem \ref{thm:SecRep} this allows the 
construction of a $C^*$ bundle $\mathscr{E}(E, \sigma)$, the \textbf{DFR bundle}\footnote{Here the nomenclature 
is due to the famous paper by Doplicher, Fredenhagen and Roberts on Quantum Spacetime, \cite{DFR}, in which this 
construction is indirectly performed to obtain the algebra of Quantum Spacetime.} of $(E, \sigma)$, such that each 
fiber if isomorphic to the Heisenberg $C^*$ algebra $\mathscr{E}_{\sigma_m}$ associated to the poison vector space 
$(E_m, \sigma_m)$. 
The algebra of sections of this bundle is a perfect locally $C^*$ algebra, the \textbf{DFR Algebra}, denoted by 
$\mathcal{E}(E,\sigma)$

This construction provides a way to obtain nontrival examples of locally $C^*$ algebras and so also of noncommutative
spaces. In case the Poisson tensor is non-degenerate it is easy to show that the Heisenberg $C^*$ algebras are simple, 
so that the space of points for the corresponding DFR algebra is precisely the original base manifold.

\subsection{"Locally Covariant Quantum Spacetime"}\label{subsec:LCQST}

Fixed a natural number $n$ we consider the category of all Lorentzian manifolds of dimension $2n$, equipped with 
isometric embeddings . Let $\sigma_0$ be the bivector associated to the standard symplectic form in $\mathbb{R}^{2n}$. 
We consider it's orbit $\Sigma$ under the action of Lorentz group $O(2n-1, 1)$.

Now given a manifold $M$ we consider the fiber bundle $M \times_{O(M, g)} \Sigma$ associated to the orthogonal 
frame bundle $O(M, g)$ with the orbit $\Sigma$ for fiber, and denote it by $\Sigma(M)$. Over $\Sigma(M)$ we 
construct a vector bundle $E^M$ by taking the pull back of the tangent bundle of $M$ under the natural projection. 

The original nondegenerate bivector $\sigma_0$ induces then a natural nondegenerate bivector field, $\sigma^M$ 
over $E^M$, defined simply by $(\xi, \xi^\prime) \mapsto \sigma (\xi, \xi^\prime)$, for $\sigma$ in the fiber 
of $\Sigma(M)$ over a given point $m\in M$ and where, by a certain abuse, we denote by $\xi$ and $\xi^\prime$ 
both vectors in  $E^{M*}_\sigma$ and their images in $T^*_m M$.

We can then use the construction outlined in the previous subsection to obtain the DFR algebra, denoted by $\mathcal{E}_\Sigma(M)$, 
associated to $(E^M, \sigma^M)$.

Any isometric embedding $\psi: M \rightarrow N$ induces a embedding, which we denote by $\psi_\Sigma$, between 
$\Sigma(M)$ and $\Sigma(N)$. Denoting by $\psi_\Sigma^*f$ the pullback of a section of $\mathscr{S}(E^N)$ by this 
embedding and by $T\psi \cdot$ the action of the differential of the original $\psi$ on the appropriated associated bundle,
we get: 

\begin{eqnarray*} \label{eq:WMOY3}
 (\psi_\Sigma^*f \star_{\sigma^M} \psi_\Sigma^*g)(\sigma)(e)&=&\int_{E^{M*}_\sigma} d\xi~e^{i\langle\xi,e\rangle}
 \int_{E^{M*}_\sigma} d\eta~\check{(\psi_\Sigma^*f)}(\sigma)(\eta) \, \check{(\psi_\Sigma^*g)}(\sigma)(\xi-\eta) \;
 e^{\frac{i}{2} \sigma(\xi,\eta)} \\
 & = & \int_{E^{M*}_\sigma} d\xi~e^{i\langle T\psi \cdot\xi,  T\psi \cdot e\rangle} \int_{E^{M*}_\sigma} d\eta~\check{f}(T\psi 
 \cdot \sigma)(T\psi \cdot\eta)  \\ 
 & & \hspace{1cm} \cdot\: \check{g}(T\psi \cdot \sigma)(T\psi \cdot(\xi-\eta)) \;
 \exp\left(\frac{i}{2} T\psi \cdot\sigma(T\psi \cdot \xi, T\psi \cdot \eta)\right) \\
 & = & \int_{E^{N*}_{T\psi \cdot \sigma}} d\xi^\prime~e^{i\langle\xi^\prime, T\psi \cdot e\rangle}
 \int_{E^{N*}_{T\psi \cdot \sigma}} d\eta^\prime~\check{(f)}(T\psi \cdot\sigma)(\eta^\prime) \\
 & & \hspace{0.9cm} \cdot\: \check{(g)}(T\psi \cdot\sigma)(\xi^\prime-\eta^\prime) \;
 \exp \left( \frac{i}{2} T\psi \cdot\sigma(\xi^\prime,\eta^\prime)\right) \\
 & = & f \star_{\sigma^N} g (T\psi \cdot \sigma) (T\psi \cdot e) \\
 & = & \psi_\Sigma^* \left(f \star_{\sigma^N} g\right) (\sigma) (e)
\end{eqnarray*}
for every $\sigma \in \Sigma(M)$ and $e \in E^M_\sigma$. 

By considering the appropriate completions this shows that the pull back $\psi_\Sigma^*$ induces a surjective 
continuous $*$-homomorphism between the DFR algebras $\mathcal{E}_\Sigma(N)$ and $\mathcal{E}_\Sigma(M)$, 
showing that $M \mapsto \mathcal{E}_\Sigma(M))$ is actually a contravariant functor between Lorentzian manifolds 
and locally $C^*$ algebras.

Moreover, since each fiber is simple, the space of points of $\mathcal{E}_\Sigma (M)$ is nothing but the 
fiber bundle $\Sigma(M)$, and the projection to the base manifold acts as a natural transformation between
the composition $\pt\circ \mathcal{E}_\Sigma$ and the inclusion of our category of manifolds into the category 
of all locally compact spaces.

Now, due to the underling bundle structure, given a Lorentzian manifold $M$ the functor $U \subset M \mapsto 
\mathcal{E}_\Sigma (U)$ defines a sheaf of locally $C^*$ algebras and so a noncommutative space in our sense. 
For a point $m \in M$ the associated stalk is isomorphic to the algebra of sections of the trivial $C^*$ bundle 
$\Sigma \times \mathscr{E}_{\sigma_0} \rightarrow \Sigma$, where $\mathscr{E}_{\sigma_0}$ is the Heisenberg $C^*$ 
algebra associated to the bivector $\sigma_0$. 

Inspired by the fact, in four dimensions, these stalks are isomorphic to the quantum spacetime algebra defined in 
the original DFR paper, we call this noncommutative space the \textbf{Quantum Spacetime} associated to the Lorentzian 
manifold $M$.

The question if this construction has any relation to the physical motivation of the original construction in 
the DFR paper will to be tackled elsewhere. 

\subsection{The Noncommutative Space defined by a Net of $C^*$ Algebras}

Another class of examples which draw motivation from mathematical physics is provided by \cite{RV:NB}. 
In that paper Ruzzi and Vasselli constructed for any given net (precosheaf) of $C^*$ algebras 
over a good basis for the topology of a\ locally compact topological space $X$, a $C_0(X)$ algebra,
such that there is a natural transformation between the functors defining the original net and
the one defining the presheaf of local sections of the $C^*$ bundle associated to the $C_0(X)$ 
algebra by theorem \ref{thm:SecRep}.

To construct the aforementioned bundle one consider, for each point $x \in X$, the restriction of the 
original net to contractible open neighborhoods of $x$ and define the algebra $\mathcal{A}_x$ as the universal algebra associated to this net. In a sense, one can interpret the observables which can be localized on such a contractible region as quantities which can be measured from the given point, so 
that $\mathcal{A}_x$  is exactly the algebra of all such observables.

Using our methods we can replace the presheaf of $C^*$ algebras used in the original paper by an actual
sheaf of locally $C^*$ algebras, such that, in case the aforementioned universal algebras have trivial
centers, the topological information about the underling space $X$ can be recovered from the algebra of
global sections of that sheaf. 

We call this sheaf the noncommutative space defined by the net of $C^*$ algebras.

This provides a interesting interpretation of those fibers as the algebras of quantities which can be
measured in a topologically trivial neighborhood around each point, so that the local sections of the
associated bundle are "consistent" choices of available observables for each point in the open region.

This result provides a connection between the usual notion of algebraic quantum field theory and our
formalism. 
Unfortunately, in general, is quite hard to provide concrete examples for this connection, since in general,
the universal algebras mentioned above are fairly hard to compute.

\section{Outlook}\label{sec:End}

As pointed out in the introduction, the main goal of this work was providing a definition of noncommutative
spaces by a generalization of Gelfand duality. 

The next obvious step in the road to noncommutative geometry is obviously to understand what a
"noncommutative smooth manifold" should be. Fortunately the commutative realm provide some
interesting tips in regards to that. 
First of all a tentative definition for a "noncommutative topological manifold" is a perfect locally $C^*$
algebra for which the associated sheaf of locally $C^*$ algebras is locally equivalent to the sheaf of
local sections of a trivial $C^*$ bundle over some $\mathbb{R}^n$, a definition which reduces to the
usual one under the assumption of commutativity.

In the usual sense a differential structure can be seen as a choice of an algebra of smooth functions over
our base manifold. Not all choices of subalgebras are allowed, the requirement that the induced sheaf of
algebras must be locally equivalent to the sheaf of smooth functions over $\mathbb{R}^n$ imposes a
series of restrictions on the nature of such a subalgebra. The first and most obvious one is that it must 
be stable under the "smooth functional calculus", the composition of two smooth functions must be again 
a smooth function. A less obvious one, discussed in \cite{GBVF} is that of spectral invariance. 
Besides, it is also clear that such a subalgebra must admit only one locally $C^*$ completion, so that 
the topological structure of the underling space is determined uniquely by this subalgebra. 
It happens that all of this proprieties are incorporated in definition of differentiable structure in a 
(possibly noncommutative) $C^*$ algebra introduced by  Blackadar and Cuntz in \cite{BC:diff} and
and latter extended by Bhatt, Inoue and Ogi in \cite{BIO:diff}. 
It is our intent to generalize this notion to locally $C^*$ algebras and study its relation to our notion 
of noncommutative space.

A different and interesting direction is to better understand the notion of locally covariant spacetime introduced in section \ref{subsec:LCQST}. As we noted before, despite de aforementioned isomorphism between the stalks of the quantum spacetime sheaf and the original quantum spacetime algebras, 
it is still not clear if any of the physical motivation for the original construction can be carried to this
extension. 
It is our intent to investigate this question in the near future.

Yet another direction which deserves further attention  are the noncommutative spaces defined by a 
net of $C^*$ algebras.  
In a series of papers started with \cite{BDS:TD}, Benini, Dappiaggi and Schenkel show that in the 
general setting introduced by \cite{BFV:LCQFT} the usual requirement of injectiveness in AQFT may 
fail to hold in presence of topological defects of the underling spacetime. 
A way out proposed by Fredenhagen is to deal away with those difficulties by restricting local covariance
to topologically trivial regions and, in this case, the construction in \cite{RV:NB} shows that the 
information about those topological defects is encoded in some cohomology theory for the resulting
algebra. 
The extension presented here simplify the methods required for their construction, and point to a
interesting new direction. 
We hope that, given appropriate conditions, one can recover both the topological and causal information
about a spacetime from a locally $C^*$ algebra and a net of $C^*$ algebras over some sub-poset of 
its set of $C^*$ seminorms, such that there is a natural transformation between the net and the sheaf
associated to the original locally $C^*$ algebra.

\end{document}